\documentclass[a4paper,11pt]{article}

\usepackage[]{graphicx}
\usepackage[]{amsmath,amssymb}
\usepackage[]{amsthm,enumerate}
\usepackage{verbatim,bm,comment}
\usepackage{color}
\usepackage{a4wide}

\newtheorem{theorem}{Theorem}[section]
\newtheorem{lemma}[theorem]{Lemma}

\newtheorem{corollary}[theorem]{Corollary}
\theoremstyle{definition}
\newtheorem{definition}[theorem]{Definition}

\newcommand{\defn}[1]{{\em #1}}
\theoremstyle{remark}
\newtheorem{remark}[theorem]{Remark}

\title{Divisible design digraphs and association schemes} 
\date{
\today
}

\author{
 Hadi Kharaghani\thanks{Department of Mathematics and Computer Science, University of Lethbridge,
Lethbridge, Alberta, T1K 3M4, Canada. \texttt{kharaghani@uleth.ca}} 
\and  
 Sho Suda\thanks{Department of Mathematics,  National Defense Academy of Japan, Yokosuka, Kanagawa 239-8686, Japan. \texttt{ssuda@nda.ac.jp}}
}
\begin{document}

\maketitle

\begin{abstract}
    Divisible design digraphs are constructed from skew balanced generalized weighing matrices and generalized Hadamard matrices.  
    Commutative and non-commutative association schemes are shown to be attached to the constructed divisible design digraphs.  
\end{abstract}

\section{Introduction}

Divisible design graphs \cite{HKM} have been received much attention as a generalization of $(v,k,\lambda)$-graphs, due to the important role they play in algebraic graph theory and design theory. 
The directed version of divisible design graphs, divisible design digraphs, was  introduced  and  studied
 extensively in establishing their feasible parameters, existence and non-existence results in \cite{CK} and in the case of Cayley graphs in \cite{CKS}. 
 A particular class of divisible design digraphs such that its adjacency matrix $A$ satisfies the condition $A+A^\top=J_v-I_m\otimes J_n$ is known to be equivalent to some association schemes with $3$ classes, and is studied in a framework of normally regular digraphs in \cite{J}. 
 In this note we deal with divisible design digraphs that do not satisfy the condition $A+A^\top=J_v-I_m\otimes J_n$.

Association schemes \cite{BI} in a way are a generalization of the centralizer algebra of the transitive permutation groups, and play an important role as the underlying space of coding theory and design theory.  
On the other hand, various design theoretic objects, such as symmetric designs, linked systems of symmetric designs,  symmetric group divisible designs, and Hadamard matrices, have been shown to produce different classes of association schemes \cite{D},  \cite{KS2017}, \cite{KS2019},  \cite{KS2020}, \cite{M}.   

The main purpose of this paper is to introduce different methods of construction for association schemes from divisible design digraphs. 
More precisely, we will pursue to construct families of divisible design digraphs from skew balanced generalized weighing matrices and generalized Hadamard matrices. 
These divisible design digraphs are based on the construction in \cite{IK}, \cite{KSS},  \cite{KS2018}, \cite{KS2019}. 
We then use the constructed families of divisible design digraphs to obtain commutative and non-commutative association schemes.

\section{Preliminaries}
As usual, we use $I_n,J_n,O_n$ to denote the identity matrix of order $n$, the all-ones matrix of order $n$, the zero matrix of order $n$, respectively. 
Let $U_n$ be a circulant matrix of order $n$ with the first row $(0,1,0,\ldots,0)$. 
Let $N_{n}$ be a negacirculant matrix with the first row $(0,1,0,\ldots,0)$ defined by 
\begin{align*}
N_{n}=\begin{pmatrix}
0 & 1 & 0 & \ldots & 0 \\
0 & 0 & 1 & \ldots & 0 \\
\vdots & \vdots & \vdots & \ddots & \vdots \\ 
0 & 0 & 0 & \ldots & 1 \\
-1 & 0 & 0 & \ldots & 0 
\end{pmatrix}. 
\end{align*}
Let $R_n=(r_{ij})$ be the back diagonal matrix of order $n$, that is $r_{ij}=\delta_{i+j,n+1}$. 
We denote by $[X]_{ij}$ the $(i,j)$-block of a block matrix $X$. 

\subsection{Divisible design digraphs}\label{sec:ddd}
A \emph{digraph} is a pair $\Gamma=(V,E)$, where $V$ is a finite nonempty set of vertices and $E$ is a set of ordered pairs (arcs) $(x,y)$ with $x,y\in V, x\neq y$. 

A digraph $\Gamma$ is called \emph{regular} of degree $k$ if each vertex of $\Gamma$ dominates exactly $k$ vertices and is dominated by exactly $k$ vertices.  
A digraph $\Gamma$ is \emph{asymmetric} if $(x,y)\in E$ implies $(y,x)\not\in E$. 
If $(x,y)$ is an arc, then we say that $x$ dominates $y$ or  $y$ is dominated by $x$. 
The \emph{adjacency matrix} of a digraph $\Gamma=(V,E)$ on $v$ vertices is a $v\times v$ $(0,1)$-matrix with rows and columns indexed by the elements $V$ such that for $x,y\in V$, $A_{xy}=1$ if $(x,y)\in E$ and $A_{xy}= 0$ otherwise. 
Note that a digraph $\Gamma$ is asymmetric  if and only if $A+A^\top$ is a $(0,1)$-matrix, and a digraph $\Gamma$ is regular of degree $k$  if and only if $AJ_v=A^\top J_v=kJ_v$ where $v$ is the number of the vertices.

\begin{definition}
Let $\Gamma$ be a regular asymmetric digraph of degree $k$ on $v$ vertices. The digraph $\Gamma$ is called a \emph{divisible design digraph with parameters $(v,k,\lambda_1,\lambda_2,m,n)$} if the vertex set can be partitioned into $m$ classes of size $n$, such that for any two distinct vertices $x$ and $y$ from the same class, the number of vertices $z$ that dominates or being dominated by both $x$ and $y$ is equal to $\lambda_1$, and for any two distinct vertices $x$ and $y$ from different classes, the number of vertices $z$ that dominates or being dominated by both $x$ and $y$ is equal to $\lambda_2$. 
\end{definition}
It is easy to see that a digraph $\Gamma=(V,E)$ is a divisible design digraph if and only if its adjacency matrix $A$ satisfies that $A+A^\top$ is a $(0,1)$-matrix and 
$$
AA^\top=A^\top A=k I_v+\lambda_1(I_m\otimes J_n-I_v) +\lambda_2(J_v-I_m\otimes J_n). 
$$
\subsection{Balanced generalized weighing matrices}\label{sec:BGW}
Let $G$ be a multiplicatively written finite group. 
A \emph{balanced generalized weighing matrix with parameters $(v, k, \lambda)$ over $G$}, or a $BGW(v, k, \lambda;G)$, is a matrix $W = (w_{ij})$
of order $v$ with entries from $G \cup \{0\}$ such that (i) every row of $W$ contains exactly
$k$ nonzero entries and (ii) for any distinct $i, h \in \{1, 2,\ldots, v\}$, every element of $G$ is
contained exactly $\lambda/|G|$ times in the multiset $\{w_{ij} w^{-1}_{hj} \mid 1 \leq j \leq v, w_{ij} \neq 0, w_{hj}\neq 0\}$.
When the group is a cyclic group generated by $U$, we write the entries of a BGW as $U^{w_{ij}}$.
A BGW $W=(U^{w_{ij}})_{i,j=1}^{n+1}$ with parameters $(n+1, n, n-1)$ over a cyclic group $G=\langle U \rangle$ of order $2m$ with diagonal $0$ is said to be \emph{skew} if $U^{w_{ji}}=U^{w_{ij}+m}$ holds for any distinct $i,j$. 
Skew BGWs will be dealt with in Sections~\ref{sec:DDD416}, \ref{sec:sbgw}. 
A BGW with $v=k$ is said to be a \emph{generalized Hadamard matrix}, which will be dealt with in Section~\ref{sec:GH}. 
\begin{lemma}[\cite{IK}]
Let $q,m$ be positive integers such that $q$ is an odd prime power, $(q-1)/m$ is odd. 
Then there is a skew $BGW(q+1,q,q-1)$ with zero diagonal over the cyclic group of order $m$.
\end{lemma}

\subsection{Association schemes}
Let $d$ be a positive integer. 
Let $X$ be a finite set of size $v$ and $R_i$ ($i\in\{0,1,\ldots,d\}$) be a nonempty subset of $X\times X$, and $A_i$ be the adjacency matrix of the graph $(X,R_i)$. 
An \emph{association scheme} with $d$ classes is a pair $(X,\{R_i\}_{i=0}^d)$, where the corresponding adjacency matrices $\{A_i\}_{i=0}^d$ to $\{R_i\}_{i=0}^d$ satisfy:

\begin{enumerate}[(i)]
\item $A_0=I_{v}$.
\item $\sum_{i=0}^d A_i = J_{v}$.
\item $A_i^\top \in\{A_1,A_2,\ldots,A_d\}$ for any $i\in\{1,2,\ldots,d\}$.
\item For all $i$ and $j$, $A_i A_j$ is a linear combination of $A_0,A_1,\ldots,A_d$.
\end{enumerate}
An association scheme is \emph{symmetric} if $A_i^\top=A_i$ holds for any $i$. 
An association scheme is \emph{commutative} if $A_iA_j=A_jA_i$ holds for any $i,j$, and \emph{non-commutative} otherwise. 
Note that if the scheme is symmetric, then it is commutative. 

We now define {\it eigenmatrices} for commutative and non-commutative cases. 
Note that for the non-commutative case, the eigenmatrices are not uniquely determined. See \cite{KS2018} for more detail and its applications to the character table.
The vector space spanned by $A_i$'s forms an algebra, denoted by $\mathcal{A}$ and is called the \emph{Bose-Mesner algebra} or \emph{adjacency algebra}.

First, assume the scheme is commutative. 
There exists a basis of $\mathcal{A}$ consisting of primitive idempotents, say $E_0=(1/|X|)J_{|X|},E_1,\ldots,E_d$. 
Note that $E_i$ is the projection onto a maximal common eigenspace of $A_0,A_1,\ldots,A_d$. 
Since  $\{A_0,A_1,\ldots,A_d\}$ and $\{E_0,E_1,\ldots,E_d\}$ are two bases in $\mathcal{A}$, there exist the change-of-bases matrices $P=(p_{ij})_{i,j=0}^d$, $Q=(q_{ij})_{i,j=0}^d$ so that
\begin{align*}
A_j=\sum_{i=0}^d p_{ij}E_i,\quad E_j=\frac{1}{|X|}\sum_{i=0}^d q_{ij}A_i.
\end{align*}
The matrices $P$ or $Q$ are said to be {\it first or second eigenmatrices} respectively. 

Second, assume that the scheme is non-commutative. 
The following is due to Higman \cite{H75}.  
Since the algebra $\mathcal{A}$ is semisimple, the adjacency algebra is isomorphic to $\oplus_{k=1}^n\text{Mat}_{d_k}(\mathbb{C})$ for uniquely determined positive integers $n,d_1,\ldots,d_n$. 
We write $\mathcal{I}=\{(i,j,k)\in\mathbb{N}^3 \mid 1\leq i,j\leq d_k, 1\leq k\leq n\}$ where $\mathbb{N}$ denotes the set of positive integers, and call the set $\mathcal{I}$ the \emph{index set of a dual basis}. 
For $k\in\{1,2,\ldots,n\}$ let $\varphi_k$ be an irreducible representation from $\mathcal{A}$ to $\text{Mat}_{d_k}(\mathbb{C})$. 
Let $E_{i,j}^{(k)}$ ($(i,j,k)\in\mathcal{I}$) be a basis of $\mathcal{A}$ such that $\varphi_k(E_{i,j}^{(k)})\in\text{Mat}_{d_k}(\mathbb{C})$, $E_{i,j}^{(k)}E_{i',j'}^{(k')}=\delta_{k,k'}\delta_{j,i'}E_{i,j'}^{(k)}$ and ${E_{i,j}^{(k)}}^*=E_{j,i}^{(k)}$, where $*$ denotes the conjugate transpose  operation.   
Since $A_\ell$ ($\ell\in\{0,1,\ldots,d\}$) and $E_{i,j}^{(k)}$ ($(i,j,k)\in\mathcal{I}$) are bases of the adjacency algebra, there exist complex numbers $p_{(i,j),\ell}^{(k)}$ and $q_{\ell,(i,j)}^{(k)}$ such that 
\begin{align*}
A_\ell=\sum_{(i,j,k)\in\mathcal{I}} p_{(i,j),\ell}^{(k)}E_{i,j}^{(k)},\quad
E_{i,j}^{(k)}=\frac{1}{v}\sum_{\ell=0}^{d} q_{\ell,(i,j)}^{(k)}A_\ell.
\end{align*}
For $k\in\{1,2,\ldots,n\}$, set $P_k=(p_{(i,j),\ell}^{(k)})$ and $Q_k=(q_{\ell,(i,j)}^{(k)})$, where $(i,j)$ runs over $\{(a,b)\mid 1\leq a,b\leq d_k\}$ and $\ell$ runs over $\{0,1,\ldots,d\}$. 
Note that the ordering of indices of rows of $P_k$ and columns of $Q_k$ is the lexicographical order. 
We then define  $(d+1)\times (d+1)$ matrices $P$ and $Q$ by 
\begin{align*}
P=\begin{pmatrix}
P_1\\
P_2\\
\vdots\\
P_n
\end{pmatrix}, \quad 
Q=\begin{pmatrix}
Q_1& Q_2& \cdots & Q_n
\end{pmatrix}.
\end{align*}

\section{Commutative association schemes with $5$ classes}\label{sec:DDD416}

Let $p=2qn+1$ be a prime power, where $q=(p-1)/(2n)$ is an odd integer and $n\geq4$ is the order of a Hadamard matrix $H$. 
Normalize $H$ to have the all-ones vector in the first row. 
For $i\in\{1,2,\ldots,n-1\}$, let $r_i$ be the $(i+1)$-th row of $H$, and define $C_i=r_i^\top r_i$.  

Define an $n^2\times n^2$ block matrix $D$ whose $(i,j)$-block of $D$ is an $n\times n$ matrix such that 
\begin{align*}
[D]_{ij}=\begin{cases}
O_{n} & \text{if } i=j,\\
C_{j-i} & \text{if } i<j,\\
-C_{j-i+n} & \text{if } i>j.
\end{cases}
\end{align*} 
Equivalently, $D$ is defined as 
\begin{align*}
D=\sum_{k=1}^{n-1}N_{n}^k \otimes C_{k}. 
\end{align*} 
Then the following is easy to see. 
\begin{lemma}\label{lem:ddd1}
$DD^\top=nI_{n}\otimes(nI_{n}-J_{n})$. 
\end{lemma}

Let $N$ be an $n^2\times n^2$ block negacirculant matrix with the first block row $(O_{n},I_{n},O_{n},\ldots,O_{n})$ defined by 
\begin{align*}
N=N_{n}\otimes I_{n}=\begin{pmatrix}
O_{n} & I_{n} & O_{n} & \ldots & O_{n} \\
O_{n} & O_{n} & I_{n} & \ldots & O_{n} \\
\vdots & \vdots & \vdots & \ddots & \vdots \\ 
O_{n} & O_{n} & O_{n} & \ldots & I_{n} \\
-I_{n} & O_{n} & O_{n} & \ldots & O_{n} 
\end{pmatrix}. 
\end{align*}
Note that the order of $N$ is $2n$. 
Set $R=R_{n}\otimes I_{n}$.  

Let $W=(N^{w_{ij}})$ be a skew $BGW(p+1,p,p-1)$ with constant diagonal $0$ over the cyclic group $\langle N \rangle$. 
Define a $(p+1)n^2\times (p+1)n^2$ block matrix $B$ whose $(i,j)$-block is an $n^2\times n^2$ matrix such that 
\begin{align*}
[B]_{ij}=\begin{cases}
O_{n^2} & \text{if }i=j, \\
D N^{w_{ij}}R & \text{if }i\neq j.
\end{cases}
\end{align*}

Decompose $B$ into disjoint $(0,1)$-matrices $A_1$ and $A_2$: $B=A_1-A_2$.   
\begin{theorem}\label{thm:ddd1}
The matrices $A_1$ and $A_2$ are the adjacency matrices of divisible design digraphs with parameters $(n^2(p+1),\frac{(n^2-n)p}{2},\frac{(n^2-2n)p}{4},\frac{(n-1)^2(p-1)}{2},p+1,n^2)$. 
\end{theorem}
\begin{proof}
First we show that $B^\top=-B$. 
The diagonal blocks of $B$ is the zero matrix of order $n^2$. 
For distinct $i,j\in\{1,2,\ldots,p+1\}$,  
\begin{align*}
    [B^\top]_{ij}&=[B]_{ji}^\top\\
    &=(D N^{w_{ji}}R)^\top\\
    &=(D (-N^{w_{ij}})R)^\top \quad\text{(by $W$ is a skew BGW and $N^{n}=-I$)}\\
    &=-R^\top(N^{w_{ij}})^\top D^\top\\
    &=-R N^{-w_{ij}} D^\top \quad\text{(by $R^\top=R$ and $N^\top=N^{-1}$)}\\
    &=-N^{w_{ij}}R D^\top \quad\text{(by $RN=N^{-1}R$)}\\
    &=-N^{w_{ij}} D R \quad\text{(by $RD^\top=DR$)}\\
    &=-DN^{w_{ij}}  R \quad\text{(by $ND=DN$)}\\
    &=-[B]_{ji}, 
\end{align*}
which proves $B^\top=-B$, and thus $A_1^\top=A_2$.  

We simultaneously calculate $A_iA_i^\top$ ($i=1,2$) and $A_iA_j^\top$ for $\{i,j\}=\{1,2\}$. 
First, we calculate $BB^\top=(A_1-A_2)(A_1^\top-A_2^\top)$ as follows. 
For $i,j\in\{1,2,\ldots,p+1\}$, 
\begin{align}
    [BB^\top]_{ij}&=\sum_{k=1}^{p+1}(1-\delta_{i,k})(1-\delta_{j,k})(D N^{w_{ik}}R)(D N^{w_{jk}}R)^\top\nonumber\\
    &=\sum_{k=1}^{p+1}(1-\delta_{i,k})(1-\delta_{j,k})D N^{w_{ik}-w_{jk}}D^\top.\label{eq:ddd1}
\end{align}
When $i=j$, \eqref{eq:ddd1} equals $npI_{n}\otimes(nI_{n}-J_{n})$ by Lemma~\ref{lem:ddd1}. 
When $i \neq j$, 
\begin{align*}
    \sum_{k=1}^{p+1}(1-\delta_{i,k})(1-\delta_{j,k})D N^{w_{ik}-w_{jk}}D^\top&=D(\sum_{k\neq i,j}N^{w_{ik}-w_{jk}})D^\top\\
    &=D(\frac{p-1}{2n}\sum_{\ell=0}^{2n-1}N^{\ell})D^\top\\
    &=O_{16n^2}, 
\end{align*}
where we used in the second last and last equalities the facts that $W=(N^{w_{ij}})$ is a BGW over $\langle N \rangle$ and $\sum_{\ell=0}^{2n-1}N^\ell=O_{n^2}$ respectively.   
Putting these into the blocks of $BB^\top$, we obtain 
\begin{align}\label{eq:ddd2}
BB^\top=npI_{p+1}\otimes I_{n}\otimes (nI_{n}-J_{n}).
\end{align}
Next we calculate $(A_1+A_2)(A_1+A_2)^\top$ as follows. 
Then $A_1+A_2=((1-\delta_{i,j})(J_{n}-U_{n}^{w_{ij}}R_{n})\otimes J_{n})_{i,j=1}^{p+1}$. 
In a similar fashion to \eqref{eq:ddd2}, we obtain  
\begin{align}\label{eq:ddd3}
    (A_1+A_2)(A_1+A_2)^\top&=npI_{p+1}\otimes (I_{n}\otimes J_{n}+(n-2)J_{n^2})\nonumber\\
    &\quad +(n-1)^2(p-1)(J_{p+1}-I_{p+1})\otimes J_{n^2}. 
\end{align}
Thirdly, it follows from the orthogonality between $C_i$ and $J_{4n}$ for $i\in\{1,2m\ldots,n-1\}$ that 
\begin{align}\label{eq:ddd4}
    (A_1+A_2)(A_1^\top-A_2^\top)=(A_1-A_2)(A_1^\top+A_2^\top)=O.
\end{align}
Putting \eqref{eq:ddd2}, \eqref{eq:ddd3}, and \eqref{eq:ddd4} together yields 
\begin{align*}
    A_1A_1^\top=A_2A_2^\top&=\frac{1}{4}n^2pI_{p+1}\otimes I_{n^2}+\frac{1}{4}n(n-2)pI_{p+1}\otimes J_{n^2}\nonumber\\
    &\quad +\frac{(n-1)^2(p-1)}{4}(J_{p+1}-I_{p+1})\otimes J_{n^2},\\
    A_1A_2^\top=A_2A_1^\top&=-\frac{1}{4}n^2pI_{p+1}\otimes I_{n^2}+\frac{1}{4}n(n-2)pI_{p+1}\otimes J_{n^2}\nonumber\\
    &\quad +\frac{(n-1)^2(p-1)}{4}(J_{p+1}-I_{p+1})\otimes J_{n^2}\nonumber\\
    &\quad +\frac{1}{2}npI_{p+1}\otimes I_{n}\otimes J_{n}.
\end{align*}
Therefore $A_1$ and $A_2$ are the adjacency matrices of divisible design digraphs with the desired parameters. 
\end{proof}

Define $A_0,A_3,A_4,A_5$ by 
\begin{align*}
A_0&=I_{n^2(p+1)},\\
A_3&=(J_{p+1}-I_{p+1})\otimes J_{n^2}-A_1-A_2,\\
A_4&=I_{p+1}\otimes I_{n}\otimes (J_{n}-I_{n}),\\
A_5&=I_{p+1}\otimes (J_{n}-I_{n})\otimes J_{n}.
\end{align*}

\begin{theorem}\label{thm:ddd2}
The set of matrices $\{A_0,A_1,\ldots,A_5\}$ forms a commutative association scheme with $5$ classes. 
\end{theorem}
\begin{proof}
The matrix $A_3$ can be written as 
$$
A_3=((1-\delta_{i,j})(U_{n}^{w_{ij}}R_{n})\otimes J_{n})_{i,j=1}^{p+1},  
$$
which shows $A_3^\top=A_3$. 

Set $\mathcal{A}=\text{span}_{\mathbb{C}}\{A_0,A_1,\ldots,A_5\}$. 
It suffices to show $A_iA_j=A_j A_i\in\mathcal{A}$ for any $i,j$. 
The cases that $i,j\in\{1,2\}$ were already shown in Theorem~\ref{thm:ddd1}, and the cases that $(i,j)\in\{(1,3),(2,3),$ $(3,1),(3,2),(3,3)\}$ follow from $A_1+A_2+A_3=(J_{p+1}-I_{p+1})\otimes J_{n^2}$.   
The cases that $i,j\in\{4,5\}$ are obvious and the remaining cases that $(i,j)\in\{1,2\}\times \{4,5\} \cup \{4,5\}\times \{1,2\}$ follow from the fact that the product of each block of $A_1,A_2,A_3$ are multiples of the all-ones matrix.  
\end{proof}
The intersection matrix $B_1=(p_{1,j}^k)_{j,k=0}^{5}$ is 
$$
B_1=\left(
\begin{smallmatrix}
 0 & 1 & 0 & 0 & 0 & 0 \\
 0 & \frac{1}{4} (n-1)^2 (p-1) & \frac{1}{4} (n-1)^2 (p-1) & \frac{1}{4} (n-1)^2 (p-1) & \frac{1}{4}n^2 p & \frac{1}{4}n (n-2) p \\
 \frac{1}{2} n (n-1) p & \frac{1}{4} (n-1)^2 (p-1) & \frac{1}{4} (n-1)^2 (p-1) & \frac{1}{4} (n-1)^2 (p-1) & \frac{1}{4}n ( n-2) p & \frac{1}{4}n (n-2) p \\
 0 & \frac{1}{2} (n-1) (p-1) & \frac{1}{2} (n-1) (p-1) & \frac{1}{2} (n-1) (p-1) & 0 & \frac{1}{2} n p \\
 0 & \frac{1}{2} n-1 & \frac{1}{2} n & 0 & 0 & 0 \\
 0 &  \frac{1}{2}n ( n-2) & \frac{1}{2}n ( n-2) & \frac{1}{2}n ( n-2) & 0 & 0 \\
\end{smallmatrix}
\right)
$$
and thus the eigenmatrices $P,Q$ are obtained as 
\begin{align*}
    P&=\left(
\begin{array}{cccccc}
 1 & \frac{1}{2} n (n-1) p & \frac{1}{2} n (n-1) p & n p & n-1 & n (n-1) \\
 1 & -\frac{1}{2} n (n-1) & -\frac{1}{2} n (n-1) & -n & n-1 &  n (n-1) \\
 1 & \frac{1}{2} n \sqrt{p} & \frac{1}{2} n \sqrt{p} & - n \sqrt{p} & n-1 & -n \\
 1 & -\frac{1}{2} n \sqrt{p} & -\frac{1}{2} n \sqrt{p} &  n \sqrt{p} &  n-1 & - n \\
 1 & \frac{1}{2}  n \sqrt{-p} & -\frac{1}{2} n \sqrt{-p} & 0 & -1 & 0 \\
 1 & -\frac{1}{2} n \sqrt{-p} & \frac{1}{2} n \sqrt{-p} & 0 & -1 & 0 \\
\end{array}
\right),\\
    Q&=\left(
\begin{array}{cccccc}
 1 & p & \frac{1}{2} (n-1) (p+1) & \frac{1}{2} (n-1) (p+1) &  \frac{1}{2}n (n-1) (p+1) & \frac{1}{2} n (n-1) (p+1) \\
 1 & -1 & \frac{p+1}{2 \sqrt{p}} & \frac{-p-1}{2 \sqrt{p}} & -\frac{ \sqrt{-1} n (p+1)}{2\sqrt{p}} & \frac{ \sqrt{-1} n (p+1)}{2\sqrt{p}} \\
 1 & -1 & \frac{p+1}{2 \sqrt{p}} & \frac{-p-1}{2 \sqrt{p}} & \frac{\sqrt{-1} n (p+1)}{2\sqrt{p}} & -\frac{\sqrt{-1} n (p+1)}{2\sqrt{p}} \\
 1 & -1 & -\frac{(n-1) (p+1)}{2 \sqrt{p}} & \frac{(n-1) (p+1)}{2 \sqrt{p}} & 0 & 0 \\
 1 & p & \frac{1}{2} (n-1) (p+1) & \frac{1}{2} (n-1) (p+1) & -\frac{1}{2} n (p+1) & -\frac{1}{2} n (p+1) \\
 1 & p & \frac{1}{2} (-p-1) & \frac{1}{2} (-p-1) & 0 & 0 \\
\end{array}
\right).
\end{align*}


\section{Non-commutative association schemes with $4m-1$ classes}\label{sec:sbgw}
In this section, we construct a non-commutative association scheme from any given skew $BGW(n+1,n,n-1)$ with zero diagonal over a cyclic group. 
The method described here is a modification of a construction based on a symmetric $BGW(n+1,n,n-1)$ with zero diagonal over a cyclic group in \cite{KS2018}. 

Let $W=(U_{2m}^{w_{ij}})_{i,j=1}^{n+1}$ be any skew $BGW(n+1,n,n-1)$ with constant diagonal zero over the cyclic group $\langle U_{2m} \rangle$ of order $2m$.

For $\ell\in\{0,1,\ldots,2m-1\}$, define $N_{\ell}$ to be a $2m(n+1)\times 2m(n+1)$ $(0,1)$-matrix as an $(n+1)\times(n+1)$ block matrix with $(i,j)$-block of order $2m$ equal to
\begin{align*} 
[N_\ell]_{ij}&=
\begin{cases}
O_{2m} & \text{ if } i=j,\\
U_{2m}^{w_{ij}+\ell} R_{2m} & \text{ if }i\neq j,
\end{cases}
\end{align*}
where 
recall that $R_{2m}$ is the back diagonal matrix of order $2m$. 
Note that $U_{2m}R_{2m}=R_{2m}U_{2m}^{-1}$.  
Then the following holds. 
\begin{theorem}\label{thm:bgw}
\begin{enumerate}
\item $\sum_{\ell=0}^{2m-1}N_{\ell}=(J_{n+1}-I_{n+1})\otimes J_{2m}$. 
\item For $\ell\in\{0,1,\ldots,2m-1\}$, $N_{\ell+m}^\top=N_\ell$, where the indices are taken modulo $2m$. 
\item For $\ell\in\{0,1,\ldots,2m-1\}$, $N_\ell(I_{n+1}\otimes J_{2m})=(I_{n+1}\otimes J_{2m})N_\ell=(J_{n+1}-I_{n+1})\otimes J_{2m}$. 
\item For $\ell,\ell'\in\{0,1,\ldots,2m-1\}$, 
\begin{align*}
N_\ell N_{\ell'}^\top=n I_{n+1}\otimes U_{2m}^{\ell-\ell'}+\frac{n-1}{2m}(J_{n+1}-I_{n+1})\otimes J_{2m}. 
\end{align*}
\end{enumerate}
In particular, $N_0,\ldots,N_{2m-1}$ are the adjacency matrices of divisible design digraphs with parameters $(2m(n+1),n,0,$ $\frac{n-1}{2m},n+1,2m)$. 
\end{theorem}
\begin{proof}
(1): For distinct $i,j\in\{1,2,\ldots,n+1\}$, the $(i,j)$-block of $\sum_{\ell=0}^{2m-1}N_\ell$ is 
$$
[\sum_{\ell=0}^{2m-1}N_\ell]_{ij}=\sum_{\ell=0}^{2m-1}[N_\ell]_{ij}=\sum_{\ell=0}^{2m-1}U_{2m}^{w_{ij}+\ell} R_{2m}=U_{2m}^{w_{ij}}J_{2m}R_{2m}=J_{2m}, 
$$
which proves (1). 

(2): 
For distinct $i,j\in\{1,2,\ldots,n+1\}$, the transpose of $(j,i)$-block of $N_{\ell+m}$ is 
\begin{align*}
(U_{2m}^{w_{ji}+\ell+m} R_{2m})^\top=(U_{2m}^{w_{ij}+\ell} R_{2m})^\top=R_{2m}U_{2m}^{-w_{ij}-\ell}=U_{2m}^{w_{ij}+\ell} R_{2m}, 
\end{align*} 
which is equal to the $(i,j)$-block of $N_{\ell}$. 
Thus $N_{\ell+m}^\top=N_\ell$ holds. 

(3): It follows from the fact that each off-diagonal block of $N_\ell$ is a permutation matrix. 

(4): 
Let $\ell,\ell'\in\{0,1,\dots,2m-1\}$. 
For $i,j\in\{1,2,\ldots,n+1\}$, we calculate the $(i,j)$-block of $N_\ell N_{\ell'}^\top$ as follows:
\begin{align*}
[N_\ell N_{\ell'}^\top]_{ij}&=\sum_{k=1}^{n+1}(1-\delta_{i,k})(1-\delta_{j,k})U_{2m}^{w_{ik}+\ell}  R_{2m}(U_{2m}^{w_{jk}+\ell'} R_{2m})^\top\\
&=\sum_{k=1}^{n+1}(1-\delta_{i,k})(1-\delta_{j,k})U_{2m}^{w_{ik}-w_{jk}} U_{2m}^{\ell-\ell'} \\
&=\begin{cases}
n U_{2m}^{\ell-\ell'} &\text{ if } i=j,\\
\frac{n-1}{2m} J_{2m} & \text{ if } i\neq j,
\end{cases}
\end{align*}
which proves the desired equality.  
\end{proof}

We now construct an association scheme from the divisible design digraphs $N_\ell$ ($\ell\in\{0,1,\ldots,2m-1\}$). 
Define
\begin{align*}
A_{\ell,0}=I_{n+1}\otimes U_{2m}^\ell,\quad A_{\ell,1}=N_\ell 
\end{align*}
for $\ell \in\{0,1,\ldots,2m-1\}$. 
Then the following is the main theorem in this section. 
\begin{theorem}\label{thm:as11}
The set of matrices $\{A_{\ell,0},A_{\ell,1}\mid \ell=0,1,\ldots,2m-1\}$ is a non-commutative association scheme with $4m-1$ classes.
\end{theorem}
\begin{proof}
The conditions (i)--(iii) are satisfied by Theorem~\ref{thm:bgw} (1), (2). 
We need to show the condition (iv) is satisfied. 
It is easy to see that  
\begin{align*}
A_{\ell,0} A_{\ell',0}=A_{\ell+\ell',0},\quad A_{\ell,0} A_{\ell',1}&=A_{\ell+\ell',1},\quad A_{\ell,1} A_{\ell',0}=A_{\ell-\ell',1}, 
\end{align*}
where the addition and subtraction of indices are taken modulo $2m$. 
Finally for $\ell,\ell'\in\{0,1,\ldots,2m-1\}$,   Theorem~\ref{thm:bgw} (4) reads 
\begin{align*}
A_{\ell,1}A_{\ell',1}=n A_{\ell-\ell',0}+\frac{n-1}{2m}(A_{0,1}+\cdots+A_{2m-1,1}),  
\end{align*}
where the subtraction of indices are taken modulo $2m$. 
Thus the condition (iv) is satisfied. 
\end{proof}

We view the cyclic group $\langle U_{2m} \rangle$ of order $2m$ as the additive group $\mathbb{Z}_{2m}$. 
Let $w=e^{2\pi\sqrt{-1}/(2m)}$. 
For $\alpha,\beta\in\mathbb{Z}_{2m}$, 
the irreducible character denoted $\chi_{\beta}$ is $\chi_{\beta}(\alpha)=w^{\alpha \beta}$. 
The \defn{character table} $K$ of the abelian group $\mathbb{Z}_{2m}$ is a $2m\times 2m$ matrix with rows and columns indexed by the elements of $\mathbb{Z}_{2m}$ with $(\alpha,\beta)$-entry equal to $\chi_{\beta}(\alpha)$. 
Note that $\chi_{\beta}(\alpha)=\chi_{\alpha}(\beta)$. 
Then the Schur orthogonality relation shows $K K^\top=2mI_{2m}$.


Define $F_{\alpha,0},F_{\alpha,1}$ as 
\begin{align*}
F_{\alpha,0}=\sum_{\gamma\in\mathbb{Z}_{2m}}\chi_\alpha(\gamma)A_{\gamma,0},\quad F_{\alpha,1}=\sum_{\gamma\in\mathbb{Z}_{2m}}\chi_{\alpha}(\gamma)A_{\gamma,1}.
\end{align*}
Using the intersection numbers described in Theorem~\ref{thm:as11}, we are able to conclude the following lemma. 
\begin{lemma}\label{lem:F}
The matrices $F_{\alpha,0},F_{\alpha,1}$ ($\alpha\in\mathbb{Z}_{2m}$) satisfy the following equations; for $\alpha,\beta\in\mathbb{Z}_{2m}$, 
\begin{align*}
F_{\alpha,0}F_{\beta,0}&=2\delta_{\alpha,\beta}mF_{\alpha,0}, \\
F_{\alpha,1}F_{\beta,1}&=2\delta_{\alpha,-\beta}nmF_{\alpha,0}+2\delta_{\alpha,0}\delta_{\beta,0}m(n-1)F_{0,1},\\
F_{\alpha,0}F_{\beta,1}&=2\delta_{\alpha,\beta}mF_{\alpha,1}, \\
F_{\alpha,1}F_{\beta,0}&=2\delta_{\alpha,-\beta}mF_{\alpha,1}. 
\end{align*}
\end{lemma}

For $i\in\{0,1,2,3\},j,k\in\{1,2\},\alpha\in \mathbb{Z}_{2m}$, let $E_i,E_{j,k}^{(\alpha)}$ be  
\begin{align*}
E_0&=\frac{1}{2(n+1)m}J_{2(n+1)m},\\
E_1&=\frac{1}{2(n+1)m}(n F_{0,0}-F_{0,1}),\\
E_2&=\frac{1}{4m}(F_{m,0}+\frac{1}{\sqrt{n}}F_{m,1}),\\
E_3&=\frac{1}{4m}(F_{m,0}-\frac{1}{\sqrt{n}}F_{m,1}),\\
E_{1,1}^{(\alpha)}&=\frac{1}{2m}F_{\alpha,0}, \quad E_{2,2}^{(\alpha)}=\frac{1}{2m}F_{-\alpha,0}, \quad E_{1,2}^{(\alpha)}=\frac{1}{2m\sqrt{n}}F_{\alpha,1}, \quad E_{2,1}^{(\alpha)}=\frac{1}{2m\sqrt{n}}F_{-\alpha,1}.   
\end{align*}

\begin{theorem}\label{thm:bgwas}
The matrices $E_0,E_1,E_2,E_3,E_{1,1}^{(\alpha)},E_{1,2}^{(\alpha)},E_{2,1}^{(\alpha)},E_{2,2}^{(\alpha)}$, $\alpha\in \{1,2,\ldots,m-1\}$, provide the Wedderburn decomposition of the adjacency algebra of the non-commutative association scheme in Theorem~\ref{thm:as11}. 
\end{theorem}
\begin{proof}
It readily follows from Lemma~\ref{lem:F}. 
\end{proof}

\begin{remark}
In Theorem~\ref{thm:bgwas}, the matrix $Q$ is determined by the matrices $Q_k$ below.  
The adjacency algebra is isomorphic to $\oplus_{k=1}^{4+(m-1)}\text{Mat}_{d_k}(\mathbb{C})$ where $(d_k)_{k=1}^{4+(m-1)}=(1,1,1,1,2,\ldots,2)$
 with  
\begin{align*}
Q_1&=\begin{pmatrix}\chi_0\\ \chi_{0}\end{pmatrix}, Q_2=\begin{pmatrix}n \chi_{0}\\ -\chi_{0}\end{pmatrix},\\
Q_3&=\frac{n+1}{2}\begin{pmatrix}\chi_{m}\\ \frac{1}{\sqrt{n}}\chi_{m}\end{pmatrix}, Q_4=\frac{n+1}{2}\begin{pmatrix}\chi_{m}\\-\frac{1}{\sqrt{n}}\chi_{m}\end{pmatrix}, \\
Q_{k+4}&=(n+1)\begin{pmatrix}\chi_k & \bm{0} & \chi_{-k} & \bm{0} \\ \bm{0} & \frac{1}{\sqrt{n}}\chi_k& \bm{0} & \frac{1}{\sqrt{n}}\chi_{-k}\end{pmatrix},
\end{align*} 
for $k=1,2,\ldots,m-1$, where $\bm{0}$ denotes the column zero vector. 
\end{remark}

Define 
\begin{align*}
    B_{0,0}&=A_{0,0},\quad     B_{\alpha,0}=A_{\alpha,0}+A_{-\alpha,0},\quad 
    B_{m,0}=A_{m,0},\\
    B_{0,1}&=A_{0,1},\quad 
    B_{\alpha,1}=A_{\alpha,1}+A_{-\alpha,1},\quad
    B_{m,1}=A_{m,1},\\
    G_{0}&=E_0,\quad G_1=E_1,\quad G_2=E_2,\quad G_3=E_3,\\
    G_{\beta,4}&=E_{1,1}^{(\beta)}+E_{2,2}^{(\beta)}+E_{1,2}^{(\beta)}+E_{2,1}^{(\beta)},\\
    G_{\beta,5}&=E_{1,1}^{(\beta)}+E_{2,2}^{(\beta)}+E_{1,2}^{(\beta)}+E_{2,1}^{(\beta)}
\end{align*}
for $\alpha,\beta\in\{1,2,\ldots,m-1\}$. 
Then the following is similarly shown. 
\begin{theorem}
The set of matrices $\{B_{\alpha,0},B_{\alpha,1}\mid \alpha=0,1,\ldots,m\}$ forms a commutative association scheme and then the set of matrices $\{G_{i},G_{\beta,4},G_{\beta,5}\mid i\in\{0,1,\ldots,5\},\beta\in\{1,2,\ldots,m-1\}\}$ with the second eigenmatrix $Q$ given as 
\begin{align*}
Q=\bordermatrix{
                 & G_0             & G_1              & G_{2}           & G_{3} & G_{\beta,4} & G_{\beta,5} \cr
B_{\alpha,0}     & 1 & n & \frac{n+1}{2}\chi_{m}(\alpha)          & \frac{n+1}{2}\chi_{m}(\alpha) & (n+1)\chi_{\beta}(\alpha) & (n+1)\chi_{\beta}(\alpha) \cr
B_{\alpha,1} & 1       & -1   & \frac{n+1}{2\sqrt{n}}\chi_{m}(\alpha)  & -\frac{n+1}{2\sqrt{n}}\chi_{m}(\alpha) & \frac{n+1}{\sqrt{n}}\chi_{m}(\beta) & -\frac{n+1}{\sqrt{n}}\chi_{m}(\beta) \cr
},
\end{align*}
where $\alpha$ runs over the set $\{0,1,\ldots,m\}$ and $\beta$ runs over the set $\{1,2,\ldots,m-1\}$. 
\end{theorem}

\section{Commutative association schemes with $3q-2$ classes
}\label{sec:GH}
In this section, we modify the construction of divisible design graphs in \cite{KS2018} to obtain divisible design digraphs attached to the finite fields of odd characteristic, and 
show that they provide commutative association schemes.  

Let $q=p^m$ be an odd prime power with $p$ an odd prime.  
We denote by $\mathbb{F}_q$ the finite field of $q$ elements, and $\mathbb{F}_q^*=\mathbb{F}_q\setminus\{0\}$.   
Let $H_q$ be the multiplicative table of $\mathbb{F}_q$, i.e., $H_q$ is a $q\times q$ matrix with rows and columns indexed by the elements of $\mathbb{F}_q$ with $(\alpha,\beta)$-entry equal to $\alpha \cdot \beta$. 
Then the matrix $H_q$ is a generalized Hadamard matrix with parameters $(q,1)$ over the additive group of $\mathbb{F}_q$. 

Let $\phi$ be a permutation representation of the additive group of $\mathbb{F}_q$ defined as follows.  
Since $q=p^m$, we view the additive group of $\mathbb{F}_q$ as $\mathbb{F}_p^m$.
Set a group homomorphism $\phi:\mathbb{F}_{p}^m\rightarrow GL_{q}(\mathbb{R})$ as $\phi((x_i)_{i=1}^m)= \otimes_{i=1}^m U_p^{x_i}$.   

From the generalized Hadamard matrix $H_q$ and the permutation representation $\phi$, we construct $q^2$ auxiliary matrices; 
for each $\alpha,\alpha'\in \mathbb{F}_q$, define a $q^2\times q^2$ $(0,1)$-matrix $C_{\alpha,\alpha'}$ to be 
\begin{align*}
C_{\alpha,\alpha'}=(\phi(\alpha(-\beta+\beta')+\alpha'))_{\beta,\beta'\in\mathbb{F}_q}. 
\end{align*}
Letting $x$ be an indeterminate, we define $C_{x,\alpha}$ by 
$$C_{x,\alpha}=\phi(\alpha)\otimes J_q$$
for $\alpha\in\mathbb{F}_q$. 

It is known that a symmetric Latin square of order $v$ with constant diagonal exists for any positive even integer $v$, see \cite{K}. 
Let $L=(L(a,a'))_{a,a'\in S}$ be a symmetric Latin square of order $q+1$ on the symbol set $S=\mathbb{F}_q\cup\{x\}$ with  constant diagonal $x$. 
Write $L$ as $L=\sum_{a\in S}a\cdot P_a$, where $P_a$ is a symmetric permutation matrix of order $q+1$. 
Note that $P_x=I_{q+1}$. 

From the $(0,1)$-matrices $C_{\alpha,\alpha'}$'s and the Latin square $L$, it is possible to construct divisible design digraphs 
as follows. 
For $\alpha \in \mathbb{F}_q$, we define\footnote{In \cite{KS2018}, the back diagonal matrix $R_{q^2}$ is post-multiplied with $C_{a,\alpha}$ so that the matrix $N_\alpha$ becomes symmetric.}  a $(q+1)q^2\times (q+1)q^2$ $(0,1)$-matrix $N_{\alpha}$ 
\begin{align*}
N_\alpha=(C_{L(a,a'),\alpha})_{a,a'\in S}=I_{q+1}\otimes\phi(\alpha)\otimes J_q+\sum_{a\in \mathbb{F}_q} P_a\otimes C_{a,\alpha}. 
\end{align*}
In order to show that each $N_{\alpha}$ is the adjacency matrix of a divisible design digraph and study more properties, we prepare a lemma on $C_{\alpha,\alpha'}$ and $P_a$.  
\begin{lemma}\label{lem:1}
\begin{enumerate}
\item For $\alpha,\alpha'\in\mathbb{F}_q$, $C_{\alpha,\alpha'}^\top=C_{\alpha,-\alpha'}$. 
\item For $a\in\mathbb{F}_q$, $\sum_{\alpha\in\mathbb{F}_q}C_{a,\alpha}=J_{q^2}$.  
\item For $\alpha\in\mathbb{F}_q$, $\sum_{a\in\mathbb{F}_q}C_{a,\alpha}=q I_q\otimes \phi(\alpha)+(J_q-I_q)\otimes J_q$. 
\item For $a\in\mathbb{F}_q$ and $\alpha,\alpha'\in\mathbb{F}_q$,
$C_{a,\alpha}C_{a,\alpha'}=q C_{a,\alpha+\alpha'}$.
\item For distinct $a,a'\in\mathbb{F}_q$ and $\alpha,\alpha'\in\mathbb{F}_q$, $C_{a,\alpha}C_{a',\alpha'}=J_{q^2}$.
\item For $a\in\mathbb{F}_q$ and $\alpha,\alpha'\in\mathbb{F}_q$, $(\phi(\alpha)\otimes J_q)C_{a,\alpha'}=C_{a,\alpha'}(\phi(\alpha)\otimes J_q)=J_{q^2}$. 
\item $\sum_{a,b\in \mathbb{F}_q,a\neq b} P_{a}P_{b}=(q-1)(J_{q+1}-I_{q+1})$. 
\end{enumerate}
\end{lemma}
\begin{proof}
(1):  
For $\alpha,\alpha',\beta,\beta'\in\mathbb{F}_q$,
\begin{align*}
[C_{\alpha,\alpha'}]_{\beta\beta'}^\top=\phi(\alpha(-\beta+\beta')+\alpha')^\top=\phi(\alpha(-\beta'+\beta)-\alpha')=[C_{\alpha,-\alpha'}]_{\beta',\beta}.
\end{align*}
Thus $C_{\alpha,\alpha'}^\top=C_{\alpha,-\alpha'}$ holds. 

(2): For $\beta,\beta'\in\mathbb{F}_q$, 
\begin{align*}
    [\sum_{\alpha\in\mathbb{F}_q}C_{a,\alpha}]_{\beta\beta'}&=\sum_{\alpha\in\mathbb{F}_q}[C_{a,\alpha}]_{\beta\beta'}=\sum_{\alpha\in\mathbb{F}_q}\phi(a(\beta-\beta')+\alpha)=\sum_{\alpha\in\mathbb{F}_q}\phi(\alpha)=J_{q}. 
\end{align*}
Thus $\sum_{\alpha\in\mathbb{F}_q}C_{a,\alpha}=J_{q^2}$ holds.
For the proof (3)-(7), see \cite[Lemma~6.1]{KS2018}. 
\end{proof}

We are now ready to focus on $N_{\alpha}$'s properties. 
\begin{theorem}\label{thm:1}
\begin{enumerate}
\item $\sum_{\alpha\in\mathbb{F}_q}N_\alpha=J_{(q+1)q^2}$. 
\item For any $\alpha\in\mathbb{F}_q$, $N_{\alpha}^\top=N_{-\alpha}$. 
\item For any $\alpha,\beta\in\mathbb{F}_q$, 
\begin{align*}
N_{\alpha} N_{\beta}&=
q^2 I_{q+1}\otimes  I_q\otimes \phi(\alpha+\beta)+q I_{q+1}\otimes (\phi(\alpha+\beta)-I_q)\otimes J_q\\
&+q I_{q+1}\otimes J_{q^2}+(q+1)J_{(q+1)q^2}.
\end{align*}
\end{enumerate}
In particular, $N_\alpha$ ($\alpha\in \mathbb{F}_q^*$) is the adjacency matrix of a divisible design digraph with parameters $((q+1)q^2,(q+1)^2,2q+1,q+1,q+1,q^2)$, and $N_0$ is that of a divisible design graph with the same parameters.  
\end{theorem}
\begin{proof}
(1): It follows from Lemma~\ref{lem:1}. 

(2): It follows from Lemma~\ref{lem:1} and the property that the matrices $P_a$ are symmetric for $a\in\mathbb{F}_q$ and $\alpha\in\mathbb{F}_q$. 

(3): We use Lemma~\ref{lem:1} to obtain:  
\begin{align*}
N_{\alpha} N_{\beta} 
&=(I_{q+1}\otimes \phi(\alpha)\otimes J_q+\sum_{a\in \mathbb{F}_q} P_a\otimes C_{a,\alpha})(I_{q+1}\otimes \phi(\beta)\otimes J_q+\sum_{b\in \mathbb{F}_q} P_b\otimes C_{b,\beta})\\
&=q I_{q+1}\otimes \phi(\alpha+\beta)\otimes J_q+\sum_{a\in \mathbb{F}_q} P_a\otimes C_{a,\alpha}( \phi(\beta)\otimes J_q)+\sum_{b\in \mathbb{F}_q}P_{b}\otimes( \phi(\alpha)\otimes J_q)C_{b,\beta}\\
&\quad +\sum_{a,b\in \mathbb{F}_q} P_{a}P_{b}\otimes C_{a,\alpha} C_{b,\beta} \\
&=qI_{q+1}\otimes \phi(\alpha+\beta)\otimes J_q+\sum_{a\in \mathbb{F}_q} P_a\otimes J_{q^2}+\sum_{b\in \mathbb{F}_q}P_{b}\otimes J_{q^2}+\sum_{a,b\in \mathbb{F}_q} P_{a}P_{b}\otimes C_{a,\alpha}C_{b,\beta}\\
&=qI_{q+1}\otimes \phi(\alpha+\beta)\otimes J_q+2(J_{q+1}-I_{q+1})\otimes J_{q^2} +\sum_{a,b\in \mathbb{F}_q} P_{a}P_{b}\otimes C_{a,\alpha}C_{b,\beta}.
\end{align*}
The third term of the above is 
\begin{align*}
&\sum_{a\in\mathbb{F}_q} P_{a}^2\otimes C_{a,\alpha}C_{a,\beta}+\sum_{a,b\in\mathbb{F}_q,a\neq b} P_{a}P_{b}\otimes C_{a,\alpha}C_{b,\beta}\\
&=\sum_{a\in\mathbb{F}_q} I_{q+1}\otimes q C_{a,\alpha+\beta}+\sum_{a,b\in\mathbb{F}_q,a\neq b} P_{a}P_{b}\otimes J_{q^2}\\
&=q I_{q+1}\otimes (q I_q\otimes \phi(\alpha+\beta)+(J_q-I_q)\otimes J_q)+(q-1)(J_{q+1}-I_{q+1})\otimes J_{q^2}\\
&=q^2 I_{q+1}\otimes I_q\otimes \phi(\alpha+\beta)+I_{q+1}\otimes J_{q^2}-qI_{q+1}\otimes I_q\otimes J_q+(q-1)J_{(q+1)q^2}. 
\end{align*}
Therefore,  
\begin{align}
N_{\alpha} N_{\beta} 
&=q^2 I_{q+1}\otimes  I_q\otimes \phi(\alpha+\beta)+q I_{q+1}\otimes (\phi(\alpha+\beta)-I_q)\otimes J_q\notag\\
&+q I_{q+1}\otimes J_{q^2}+(q+1)J_{(q+1)q^2}.\label{eq:gh1} 
\end{align}
This completes the proof. 
\end{proof}

We define $(0,1)$-matrices $A_{\alpha,i}$ ($\alpha\in\mathbb{F}_q,i\in\{0,1,2\}$) as 
\begin{align*}
A_{\alpha,0}&=I_{q(q+1)}\otimes \phi(\alpha),\\
A_{\alpha,1}&=I_{q+1}\otimes \phi(\alpha)\otimes J_q,\\
A_{\alpha,2}&=N_{\alpha}-I_{q+1}\otimes \phi(\alpha)\otimes J_q.
\end{align*}
Note that 
$A_{0,0}=I_{(q+1)q^2}$ and $\sum_{\alpha\in \mathbb{F}_q}A_{\alpha,0}=A_{0,1}$.  
\begin{theorem}\label{thm:as}
The set of matrices $\{A_{\alpha,0},A_{\beta,1},A_{\alpha,2}\mid \alpha\in\mathbb{F}_q,\beta\in\mathbb{F}_q^*\}$ forms a commutative association scheme with $3q-2$ classes.
\end{theorem}
\begin{proof}
The matrices $A_{\alpha,0}$'s and $A_{\alpha,1}$'s are non-zero $(0,1)$-matrices.  
By the definition of $N_{\alpha}$,  $A_{\alpha,2}$'s are non-zero $(0,1)$-matrices. And it holds that $\sum_{\alpha\in\mathbb{F}_q}(A_{\alpha,0}+A_{\alpha,2})+\sum_{\beta\in\mathbb{F}_q^*}A_{\beta,1}=J_{(q+1)q^2}$. 
For $\alpha\in \mathbb{F}_q$ and $i\in\{0,1,2\}$, $A_{\alpha,i}^\top=A_{-\alpha,i}$. 
Next we will show that $\mathcal{A}:=\text{span}_{\mathbb{C}}\{A_{\alpha,i} \mid \alpha\in\mathbb{F}_q,i\in\{0,1,2\}\}$ is closed under the matrix multiplication and the matrices are commuting.\footnote{Note that we involve $A_{0,1}$ in the space $\mathcal{A}$ so that the description of the proof of (iv) becomes easier.}  
For $\alpha,\beta\in\mathbb{F}_q$, the following are easy to see: 
\begin{align*}
A_{\alpha,0}A_{\beta,0}&=A_{\beta,0}A_{\alpha,0}=A_{\alpha+\beta,0},\\
A_{\alpha,0}A_{\beta,1}&=A_{\beta,1}A_{\alpha,0}=A_{\beta,1},\\
A_{\alpha,0}A_{\beta,2}&=A_{\beta,2}A_{\alpha,0}=A_{\alpha+\beta,2},\\
A_{\alpha,1}A_{\beta,1}&=q A_{\alpha+\beta,1}.
\end{align*}
By Lemma~\ref{lem:1}(6), 
\begin{align*}
A_{\alpha,1}A_{\beta,2}&=A_{\beta,2} A_{\alpha,1}=(J_{q+1}-I_{q+1})\otimes J_{q^2}. 
\end{align*}
Finally by Theorem~\ref{thm:1}(3), 
\begin{align*}
N_{\alpha}N_{\beta}=q^2 A_{\alpha+\beta,0}+q(A_{\alpha+\beta,1}-\sum_{\alpha\in \mathbb{F}_q}A_{\alpha,0})+q\sum_{\alpha\in\mathbb{F}_q}A_{\alpha,1}+(q+1)J_{(q+1)q^2}. 
\end{align*}
Now we calculate $A_{\alpha,2}A_{\beta,2}$: 
\begin{align*}
A_{\alpha,2}A_{\beta,2}&=(N_{\alpha}-A_{\alpha,1})(N_{\beta}-A_{\beta,1})\\
&=N_{\alpha}N_{\beta}-N_{\alpha}A_{\beta,1}-A_{\alpha,1} N_{\beta}+A_{\alpha,1}A_{\beta,1}\\
&=N_{\alpha}N_{\beta}-2(J_{q+1}-I_{q+1})\otimes J_{q^2}+A_{\alpha+\beta,1},
\end{align*} 
from which it holds that $A_{\alpha,2}A_{\beta,2}\in\mathcal{A}$. 
Therefore, $\mathcal{A}$ is closed under matrix multiplication and the matrices are commuting. This completes the proof.  
\end{proof}

To describe the primitive idempotents,  let $F_{\alpha,i}$ ($\alpha\in\mathbb{F}_q,i\in\{0,1,2\}$) be 
\begin{align*}
F_{\alpha,i}&=\sum_{\gamma\in\mathbb{F}_q} \chi_{\alpha}(\gamma)A_{\gamma,i}.  
\end{align*}

\begin{lemma}\label{lemma:ghf}
The following hold. 
\begin{align*}
        F_{\alpha,0}F_{\beta,0}& =\delta_{\alpha,\beta}qF_{\alpha,0},\\ 
        F_{\alpha,0}F_{\beta,1}& =F_{\beta,1}F_{\alpha,0}=\delta_{\alpha,0} q F_{\beta,1},\\ 
        F_{\alpha,0}F_{\beta,2}& =F_{\beta,2}F_{\alpha,0}=\delta_{\alpha,\beta}qF_{\alpha,2},\\ 
        F_{\alpha,1}F_{\beta,1}& =\delta_{\alpha,\beta}q^2F_{\alpha,1},\\ 
        F_{\alpha,1}F_{\beta,2}& =F_{\beta,2}F_{\alpha,1}=\delta_{\alpha,0}\delta_{\beta,0}q^2 F_{0,2},\\ 
        F_{\alpha,2}F_{\beta,2}& =\delta_{\alpha,\beta}q^3F_{\alpha,0}+\delta_{\alpha,0}\delta_{\beta,0}q^2(-qF_{0,0}+qF_{0,1}+(q-1)F_{0,2}). 
\end{align*}

\end{lemma}

For $i\in\{0,1,2\}, j\in\{2,3,4\},\alpha\in\mathbb{F}_q^*$, let $E_{i},E_{\alpha,j}$ be 
\begin{align*}
E_0&=\frac{1}{(q+1)q^2}J_{(q+1)q^2}=\frac{1}{(q+1)q^2}(F_{0,1}+F_{0,2}),\\
E_1&=\frac{1}{(q+1)q^2}(q F_{0,1}-F_{0,2}),\\
E_{\beta,2}&=\frac{1}{(q+1)q^2}(q+1)F_{\beta,1},\\
E_{\beta,3}&=\frac{1}{(q+1)q^2}\left(\frac{q(q+1)}{2} F_{\beta,0}+\frac{q+1}{2}F_{\beta,2}\right),\\
E_{\beta,4}&=\frac{1}{(q+1)q^2}\left(\frac{q(q+1)}{2} F_{\beta,0}-\frac{q+1}{2}F_{\beta,2}\right).
\end{align*}
\begin{theorem}\label{thm:pi}
The matrices $E_0,E_1,E_{\beta,2},E_{\beta,3},E_{\beta,4}$ ($\beta\in\mathbb{F}_q^*$) are the primitive idempotents of the adjacency algebra of the association scheme. 
\end{theorem}
\begin{proof}
It follows from Lemma~\ref{lemma:ghf}. 
\end{proof}
The second eigenmatrix $Q$ is written as
\begin{align*}
Q=\bordermatrix{
                 & E_0             & E_1              & E_{\beta,2}           & E_{\beta,3} & E_{\beta,4} \cr
A_{\alpha,0} & 1       & q     & q+1   &\frac{q(q+1)}{2}\chi_{\alpha}(\beta) & \frac{q(q+1)}{2}\chi_{\alpha}(\beta) \cr
A_{\alpha',1} & 1 & q & (q+1)\chi_{\alpha'}(\beta)          & 0 & 0\cr
A_{\alpha,2} & 1       & -1     & 0  &\frac{q+1}{2}\chi_{\alpha}(\beta) & \frac{-q-1}{2}\chi_{\alpha}(\beta)\cr
},\end{align*}
where $\alpha\in \mathbb{F}_q,\alpha',\beta\in\mathbb{F}_q^*$. 

\section*{Acknowledgments}
The authors would like to thank the referees for their careful reading and helpful comments. 
Hadi Kharaghani is supported by the Natural Sciences and 
Engineering  Research Council of Canada (NSERC).  Sho Suda is supported by JSPS KAKENHI Grant Number 18K03395.


\begin{thebibliography}{99}
\bibitem{BI}
E. Bannai, T. Ito, Algebraic Combinatorics I: Association Schemes,
{Benjamin/Cummings, Menlo Park, CA,} 1984.

\bibitem{CK}
D. Crnkovi\'{c} and H. Kharaghani, 
Divisible design digraphs, 
Algebraic design theory and Hadamard matrices, 43--60, Springer Proc. Math. Stat., 133, Springer, Cham, 2015.

\bibitem{CKS} 
D. Crnkovi\'{c} and H. Kharaghani, and A. \v{S}vob, 
Divisible design Cayley digraphs, 
{\sl Discrete Math.} {\bf 343} (2020), no. 4, 111784, 8 pp.

\bibitem{D}
E. van Dam, 
Three-class association schemes, 
{\sl J. Algebraic Combin.} {\bf 10} (1999), 69--107.

\bibitem{HKM}
W.H. Haemers, H. Kharaghani, and M. Meulenberg, 
Divisible design graphs, 
{\sl J. Combin. Theory Ser. A} {\bf 118} (2011) 978--992.

\bibitem{H75}
D. G. Higman, 
Coherent configurations, Part 1: Ordinary representation theory,  
{\sl Geom. Dedicata} {\bf 4} (1975), 1--32. 

\bibitem{IK}
Yury J. Ionin and H. Kharaghani, 
Doubly regular digraphs and symmetric designs, 
{\sl J. Combin. Theory Ser. A} {\bf 101} (2003) 35--48.

\bibitem{J}
L. K. J{\o}rgensen, 
Normally regular digraphs, 
{\sl Elec.\ J.\ Combin.} {\bf 22} (2015), P4.21.


\bibitem{K}
H. Kharaghani, 
New class of weighing matrices, 
{\sl Ars.\ Combin.} {\bf 19} (1985), 69--72. 

\bibitem{KSS}
H. Kharaghani, S. Sasani and S. Suda, 
A strongly regular decomposition of the complete graph and its association scheme, 
{\sl Finite Fields Appl.} {\bf 48} (2017), 356--370.

\bibitem{KS2017}
H. Kharaghani and S. Suda, 
Linked systems of symmetric group divisible designs, {\sl J. Algebraic Combin.}  {\bf 47} (2017), no. 2, 319--343.

\bibitem{KS2018}
H. Kharaghani and S. Suda, 
Non-commutative association schemes and their fusion association schemes,  
{\sl Finite Fields Appl.} {\bf 52} (2018), 108--125. 

\bibitem{KS2019}
H. Kharaghani and S. Suda, Linked system of symmetric group divisible designs of type II, {\sl Des.\ Codes  Cryptogr.} {\bf 87} (2019), no. 10, 2341--2360. 


\bibitem{KS2020}
H. Kharaghani and S. Suda, 
Commutative association schemes obtained from twin prime powers, Fermat primes, Mersenne primes, 
{\sl Finite Fields Appl.} {\bf 63} (2020), 101631, 22 pp.



\bibitem{M}
R. Mathon, 
The systems of linked $2$-$(16,6,2)$ designs, 
{\sl Ars Combin.} {\bf 11} (1981), 131--148. 
  	
\end{thebibliography}
\end{document}